\documentclass[10pt,reqno]{amsart}
\usepackage{floatflt}
\usepackage{xypic}
\usepackage{rotating}
\usepackage{bbm}
\usepackage{mathrsfs}
\usepackage{diagbox}
\usepackage{dutchcal}
\usepackage{extarrows}
\usepackage{cite}
\usepackage{amsfonts} 
\usepackage[dvipsnames,usenames]{color}
\usepackage{graphicx,latexsym,bm,amsmath,amssymb,verbatim,multicol,lscape}
\makeatletter
\@namedef{subjclassname@2020}{
\textup{2020} Mathematics Subject Classification}
\makeatother
\newtheorem{theorem}{Theorem} [section]

\newtheorem{conjecture}[theorem]{Conjecture}
\newtheorem{lemma}[theorem]{Lemma}

\newtheorem{definition}[theorem]{Definition}

\numberwithin{equation}{section}

\def\bew{\begin{widetext}}
\def\eew{\end{widetext}}
\def\be{\begin{equation}}
\def\ee{\end{equation}}
\def\bea{\begin{eqnarray}}
\def\eea{\end{eqnarray}}

\allowdisplaybreaks
\begin{document}
\title[Proof of the Gawron-Miska-Ulas conjecture concerning unboundedness]
{Proof of the Gawron-Miska-Ulas conjecture concerning unboundedness of coefficients
of power series expansion of $\prod_{n=0}^{\infty}(1-x^{2^{n}})^m$}

\begin{abstract}
It is well known that $F(x)=\prod_{n=0}^{\infty}(1-x^{2^n})$ is the generating
function of the Prouhet-Thue-Morse sequence $\{(-1)^{\sigma_2(n)}\}_{n=0}^\infty$,
where $\sigma_2(n)$ is the sum of (binary) digits of $n$. Let $m$ be an integer.
In 2018, Gawron, Miska and Ulas initiated the study of arithmetic properties of
power series expansion of the function
$$F_m(x)=F(x)^m=\sum_{n=0}^{\infty}t_m(n) x^n,$$
and proposed a conjecture stating that for any given integer $m\ge 2$,
the sequence $\{t_m(n)\}_{n=0}^{\infty}$ is unbounded. In this paper,
we introduce a new method to investigate this conjecture. In fact,
by making use of algebraic, $p$-adic and analytic methods, we show
that the Gawron-Miska-Ulas conjecture is true.
\end{abstract}

\author[J.M. Yu]{Jinmin Yu}
\address{Mathematical College, Sichuan University, Chengdu 610064, P.R. China}
\email{jmyumath@163.com}
\author[W.Z. Lei]{Wenzhong Lei}
\address{Mathematical College, Sichuan University, Chengdu 610064, P.R. China}
\email{lwzh1729@163.com}
\author[S.F. Hong]{Shaofang Hong$^*$}
\address{Mathematical College, Sichuan University, Chengdu 610064, P.R. China}
\email{sfhong@scu.edu.cn}

\thanks{$^*$S.F. Hong is the corresponding author and was supported
partially by National Science Foundation of China \#12571007.}

\maketitle

\section{Introduction}

Let $\mathbb{N}$ and $\mathbb{Z}$ denote the set of all the nonnegative integers
and the ring of integers, respectively. For any nonnegative integer $n$, we denote
by $\sigma_2(n)$ the sum of (binary) digits of $n$. Namely,
$\sigma_2(n):=\sum_{k=0}^m\epsilon_k$ if $n=\sum_{k=0}^m\epsilon_k 2^k$ with
$\epsilon_k\in\{0,1\}$, is the unique expansion of $n$ in base 2.
Furthermore, we let $\upsilon_2(n)$ stand for the 2-adic additive valuation
of $n$, i.e., $v_2(n)$ is the biggest nonnegative integer $r$ with $2^r$
dividing $n$. The Prouhet-Thue-Morse sequence (the PTM sequence for short)
on the alphabet $\{-1, +1\}$ is defined as ${\bf a}=\{a_n\}_{n=0}^\infty$,
where $a_n:=(-1)^{\sigma_2(n)}$. The sequence ${\bf a}$ satisfies the
following recurrence relations:
$$a_0=1, \ a_{2n}=a_n \ {\rm and} \ a_{2n+1}=-a_n$$
for any integer $n\ge 0$. The PTM sequence has many striking properties and lots of
applications were found in number theory, combinatorics on words, analysis on manifolds
and even physics \cite{[AS]}. The sequence ${\bf a}$ holds the simple shape
of the generating function $F(x):=\sum_{n=0}^\infty a_n x^n\in\mathbb{Z}[[x]]$.
In fact, the recurrence relations give us the functional equation
$$F(x)=(1-x)F(x^2)$$
from which one can easily derive the following identity
$$
F(x)=\prod_{n=0}^\infty(1-x^{2^n}).
$$

For the reciprocal function $\frac{1}{F(x)}$, let ${\bf b}=(b_n)_{n=0}^\infty$
be the sequence of coefficients of its power series expansion
$$
\frac{1}{F(x)}=\prod_{n=0}^\infty\frac{1}{1-x^{2^n}}=\sum_{n=0}^\infty b_n x^n\in\mathbb{Z}[[x]].
$$
The sequence ${\bf b}$ has strong combinatorial properties. In fact, it satisfies that $b_0=b_1=1$ and
$$
b_{2n}=b_{2n-1}+b_n, b_{2n+1}=b_{2n}
$$
for any integer $n\ge 1$. One call ${\bf b}$ the sequence of the binary partition
function. It was introduced by Euler and was studied by Churchhouse \cite{[Churchhouse]}
(we can also consult the papers \cite{[de Bruijn],[Knuth],[Mahler]}). It is remarked
that Gupta \cite{[Gupta]} confirmed the Churchhouse conjecture \cite{[Churchhouse]}
by proving that for any integer $k\ge 1$ and odd $n$, one has
$$b_{2^{2k+2}n}-b_{2^{2k}n}\equiv 0\pmod {2^{3k+2}}$$
and
$$b_{2^{2k+1}n}-b_{2^{2k-1}n}\equiv 0\pmod {2^{3k}}.$$
In 2019, the sequence ${\bf b}$ was generalized by Ulas and \'{Z}mija \cite{[UZ]}
to a more general sequence.

For any integer $m$, let $t_m(n)$ be the coefficient of $x^n$ of the power series
$$F_m(x)=F(x)^m:=\prod_{i=0}^{\infty}(1-x^{2^i})^m.$$
That is,
$$F_m(x)=\prod_{i=0}^{\infty}(1-x^{2^i})^m=\sum_{n=0}^{\infty}t_m(n)x^n.$$
One can easily see that both ${\bf a}$ and ${\bf b}$ are the sequences of coefficients
of the power series expansion of $F_m(x)=F(x)^m$ for $m=1$ and $m=-1$, respectively.
Evidently, one has the functional equation
$$F_m(x)=(1-x)^mF_m(x^2).$$
Moreover, for any given positive integer $m$, the sequence $\{t_m(n)\}_{n=0}^\infty$
is the Cauchy convolution of $m$ copies of the Prouhet-Thue-Morse sequence, that is
$$
t_m(n)=\sum_{(i_1,...,i_m)\in\mathbb{N}^k \atop i_1+i_2+...+i_m=n}
(-1)^{\sum_{k=1}^m\sigma_2(i_k)}.
$$
For negative integer $m$, the $n$-th term of the sequence $\{t_m(n)\}_{n=0}^\infty$
counts the number of representations of the number $n$ as a sum
of powers of 2 where each summand can have one among $-m$ colors.

In 2018, Gawron, Miska and Ulas \cite{[MG]} initiated the investigation of the arithmetic
properties of coefficients of the power series expansions of the function $F_m(x)$.
In \cite{[MG]}, Gawron, Miska and Ulas established several interesting and important results
about the properties of the power series expansions of this function. For example, they
obtained a characterization on the 2-adic valuation of the sequence $\{t_m(n)\}_{n=0}^\infty$
for $m$ being a power of 2 and $m=3$. Regarding the sequence $\{t_2(n)\}_{n=0}^\infty$, it was
proved in \cite{[MG]} that the set of its values is just $\mathbb{Z}\setminus\{0\}$ and
that it is log-concave:
$$t_2(n)^2>t_2(n-1) t_2(n+1)$$
for any positive integer $n$. It was observed in \cite{[MG]} that the crude estimation
using the fact that $|t_1(n)|=1$ gives only the equality $t_m(n)=O(n^m)$
and if $m\ge 2$, then $t_m(n)=O(n^{\frac{m}{2}})$ for any integer $n\ge 1$.
The latter one shows that there is a lot of cancellation in the sum defining
$t_m(n)$ and so one naturally asks whether the sequence
$\{t_m(n)\}_{n=0}^\infty$ is bounded or not. Gawron, Miska and Ulas \cite{[MG]}
believed that the sequence $\{t_m(n)\}_{n=0}^{\infty}$ is unbounded for any integer
$m\ge 2$. That is, they proposed the following conjecture.

\begin{conjecture} (Conjecture 3.9 of \cite{[MG]})
For each integer $m\ge 2$, we have
$$\limsup_{n\to\infty} t_m(n)=+\infty$$
and
$$
\liminf_{n\to\infty} t_m(n)=-\infty.
$$
\end{conjecture}
\noindent It was confirmed in \cite{[MG]} the truth of Conjecture 1.1 for the
case $m=3$ and $m=2^k$. But the general form of Conjecture 1.1 is still open.

In this paper, our main goal is to study the unboundedness of the sequence
$\{t_m(n)\}_{n=0}^{\infty}$ when $m\ge 2$. We introduce a new method to
investigate the problem of unboundedness of the sequence $\{t_m(n)\}_{n=0}^{\infty}$
when $m\ge 2$. Actually, by using algebraic, $p$-adic and analytic methods,
we solve this question by showing the following main result of this paper.   
\begin{theorem}
Let $m\ge 2$ be an integer. Then the sequence $\{t_m(n)\}_{n=0}^{\infty}$
is unbounded.
\end{theorem}
\noindent By Theorem 1.2, we know immediately that the Gawron-Miska-Ulas
conjecture (Conjecture 1.1) is true.

This paper is organized as follows. In Section 2, we prove some preliminary
lemmas which are needed in the proof of Theorem 1.2. In Section 3, we provide
the proof of Theorem 1.2 as the conclusion of the paper.

\section{Auxiliary lemmas}

In this section, we present several auxiliary lemmas
that are needed in the proof of Theorem 1.1.

For any positive integer $m$, we define four matrices $A_m^{(1)},
A_m^{(2)}, B_m^{(1)}$ and $B_m^{(2)}$ of order $m-1$ as follows:
If $2\nmid m$, then we define
$$
\begin{small}
A_m^{(1)}:=\left(\begin{array}{ccccccccccc}
-\binom{m}{1}&-\binom{m}{3}& -\binom{m}{5} &...&-\binom{m}{m-2}
&-\binom{m}{m}&0& 0&...&0&0\\
\binom{m}{0}&\binom{m}{2}& \binom{m}{4} &...&\binom{m}{m-3}
&\binom{m}{m-1}& 0& 0&...&0&0\\
0 &-\binom{m}{1}&-\binom{m}{3}&...&-\binom{m}{m-4}
&-\binom{m}{m-2}&-\binom{m}{m}&0&...&0 & 0\\
0 &\binom{m}{0}&\binom{m}{2}&...& \binom{m}{m-5}
&\binom{m}{m-3}&\binom{m}{m-1}& 0&...&0 &0 \\
\cdot&\cdot&\cdot&...&\cdot&\cdot&\cdot&\cdot&...&\cdot&\cdot\\
\cdot&\cdot&\cdot&...&\cdot&\cdot&\cdot&\cdot&...&\cdot&\cdot\\
\cdot&\cdot&\cdot&...&\cdot&\cdot&\cdot&\cdot&...&\cdot&\cdot\\
0&0&0&...&-\binom{m}{1} &-\binom{m}{3}&-\binom{m}{5}& -\binom{m}{7}
&...& -\binom{m}{m-2} & -\binom{m}{m}\\
0&0&0&...&\binom{m}{0} &\binom{m}{2}&\binom{m}{4}
&\binom{m}{6}&...& \binom{m}{m-3} &\binom{m}{m-1}\\
\end{array}
\right)
\end{small}
$$
and
$$
B_m^{(1)}:=\left(\begin{array}{ccccccccccc}
\binom{m}{1}&\binom{m}{3}& \binom{m}{5}&...&
\binom{m}{m-4} &\binom{m}{m-2}&\binom{m}{m}&0&...&0&0\\
0 &\binom{m}{1}&\binom{m}{3}&...& \binom{m}{m-6}
&\binom{m}{m-4} &\binom{m}{m-2}&\binom{m}{m}&...&0&0\\
\cdot&\cdot&\cdot&...&\cdot&\cdot&\cdot&\cdot&...&\cdot&\cdot\\
\cdot&\cdot&\cdot&...&\cdot&\cdot&\cdot&\cdot&...&\cdot&\cdot\\
\cdot&\cdot&\cdot&...&\cdot&\cdot&\cdot&\cdot&...&\cdot&\cdot\\
0&0&0 &...&\binom{m}{1}&\binom{m}{3}&\binom{m}{5}&\binom{m}{7}&...
&\binom{m}{m}&0\\
0&0&0 &...&0&\binom{m}{1}&\binom{m}{3}&\binom{m}{5}&...
&\binom{m}{m-2}&\binom{m}{m}\\
\binom{m}{0}&\binom{m}{2}& \binom{m}{4}&...&
\binom{m}{m-5} &\binom{m}{m-3}&\binom{m}{m-1}&0&...&0&0\\
0 &\binom{m}{0}&\binom{m}{2}&...& \binom{m}{m-7}
&\binom{m}{m-5} &\binom{m}{m-3}&\binom{m}{m-1}&...&0&0\\
\cdot&\cdot&\cdot&...&\cdot&\cdot&\cdot&\cdot&...&\cdot&\cdot\\
\cdot&\cdot&\cdot&...&\cdot&\cdot&\cdot&\cdot&...&\cdot&\cdot\\
\cdot&\cdot&\cdot&...&\cdot&\cdot&\cdot&\cdot&...&\cdot&\cdot\\
0&0&0 &...&\binom{m}{0}&\binom{m}{2}&\binom{m}{4}&\binom{m}{6}&...
&\binom{m}{m-1}&0\\
0&0&0 &...&0&\binom{m}{0}&\binom{m}{2}&\binom{m}{4}&...
&\binom{m}{m-3}&\binom{m}{m-1}\\
\end{array}
\right).
$$
And if $2|m$, then we define
$$
\begin{small}
A_m^{(2)}:=\left(\begin{array}{ccccccccccc}
-\binom{m}{1}&-\binom{m}{3}& -\binom{m}{5} &...&-\binom{m}{m-3}
&-\binom{m}{m-1}&0& 0&...&0&0\\
\binom{m}{0}&\binom{m}{2}& \binom{m}{4} &...&\binom{m}{m-4}
&\binom{m}{m-2}&\binom{m}{m}& 0&...&0&0\\
 0 &-\binom{m}{1}&-\binom{m}{3}&...&-\binom{m}{m-5} &-\binom{m}{m-3}&-\binom{m}{m-1}&0&...&0 & 0\\
 0 &\binom{m}{0}&\binom{m}{2}&...& \binom{m}{m-6}&\binom{m}{m-4}&\binom{m}{m-2}&\binom{m}{m}&...&0 &0 \\
\cdot&\cdot&\cdot&...&\cdot&\cdot&\cdot&\cdot&...&\cdot&\cdot\\
\cdot&\cdot&\cdot&...&\cdot&\cdot&\cdot&\cdot&...&\cdot&\cdot\\
\cdot&\cdot&\cdot&...&\cdot&\cdot&\cdot&\cdot&...&\cdot&\cdot\\
0&0&0&...&-\binom{m}{1} &-\binom{m}{3}&-\binom{m}{5}& -\binom{m}{7}
&...& -\binom{m}{m-1} &0\\
0&0&0&...&\binom{m}{0} &\binom{m}{2}&\binom{m}{4}& \binom{m}{6}
&...& \binom{m}{m-2} &\binom{m}{m}\\
0&0&0&...&0&-\binom{m}{1}&-\binom{m}{3}& -\binom{m}{5}
&...& -\binom{m}{m-3} &-\binom{m}{m-1}\\
\end{array}
\right)
\end{small}
$$
and
$$
B_m^{(2)}:=\left(\begin{array}{cccccccccccc}
\binom{m}{1}&\binom{m}{3}& \binom{m}{5}&...&
\binom{m}{m-5} &\binom{m}{m-3}&\binom{m}{m-1}&0&0&...&0&0\\
0 &\binom{m}{1}&\binom{m}{3}&...& \binom{m}{m-7}
&\binom{m}{m-5} &\binom{m}{m-3}&\binom{m}{m-1}&0&...&0&0\\
\cdot&\cdot&\cdot&...&\cdot&\cdot&\cdot&\cdot&\cdot&...&\cdot&\cdot\\
\cdot&\cdot&\cdot&...&\cdot&\cdot&\cdot&\cdot&\cdot&...&\cdot&\cdot\\
\cdot&\cdot&\cdot&...&\cdot&\cdot&\cdot&\cdot&\cdot&...&\cdot&\cdot\\
0&0&0 &...&0&\binom{m}{1}&\binom{m}{3}&\binom{m}{5}&\binom{m}{7}&...
&\binom{m}{m-1}&0\\
0&0&0 &...&0&0&\binom{m}{1}&\binom{m}{3}&\binom{m}{5}&...
&\binom{m}{m-3}&\binom{m}{m-1}\\
\binom{m}{0}&\binom{m}{2}& \binom{m}{4}&...&\binom{m}{m-6}
&\binom{m}{m-4} &\binom{m}{m-2}&\binom{m}{m}& 0&...&0&0\\
0 &\binom{m}{0}&\binom{m}{2}&...& \binom{m}{m-8}
&\binom{m}{m-6} &\binom{m}{m-4}&\binom{m}{m-2}&\binom{m}{m}&...&0&0\\
\cdot&\cdot&\cdot&...&\cdot&\cdot&\cdot&\cdot&\cdot&...&\cdot&\cdot\\
\cdot&\cdot&\cdot&...&\cdot&\cdot&\cdot&\cdot&\cdot&...&\cdot&\cdot\\
\cdot&\cdot&\cdot&...&\cdot&\cdot&\cdot&\cdot&\cdot&...&\cdot&\cdot\\
0&0&0 &...&\binom{m}{0}&\binom{m}{2}&\binom{m}{4}&\binom{m}{6}&
\binom{m}{8}&...&\binom{m}{m}&0\\
0&0&0 &...&0&\binom{m}{0}&\binom{m}{2}&\binom{m}{4}
&\binom{m}{6}&...&\binom{m}{m-2}&\binom{m}{m}\\
\end{array}
\right).
$$

For any integer $m\ge 2$, we define the matrices
$A_m$ and $B_m$ of order $m-1$ as follows:
\begin{equation}\label{eq2.1}
A_m:=
\begin{cases}
A_m^{(1)}, \text{if}\; 2\nmid m,\\
A_m^{(2)}, \text{if}\; 2|m\\
\end{cases}
\end{equation}
and
\begin{equation}
B_m:=\begin{cases}
B_m^{(1)}, \text{if}\; 2\nmid m,\\
B_m^{(2)}, \text{if}\; 2|m,\\
\end{cases}
\end{equation}
and define two polynomials $f_m(x)$ and $g_m(x)$ by
$$f_m(x):=\sum_{i=0 \atop 2\nmid i}^m\binom{m}{i}x^{\frac{i-1}{2}}$$
and
$$g_m(x):=\sum_{i=0 \atop 2|i}^m\binom{m}{i}x^{\frac{i}{2}}.$$
Then $\det B_m$ is the {\it eliminant} of $f_m(x)$ and $g_m(x)$.
Now we give the following two results.

\begin{lemma} \cite{[SL]} Let $m\ge 2$ be an integer. Then it holds that
$\det B_m\not=0$ if and only if $\gcd(f_m(x),g_m(x))=1$.
\end{lemma}

\begin{lemma}
Let $m\ge 2$ be an integer. Then $|\det A_m|=|\det B_m|\not=0$.
\end{lemma}
\begin{proof}
First of all, we have
\begin{equation}\label{eq2.3}
g_m(x^2)-xf_m(x^2)=(1-x)^m.
\end{equation}
We claim that $\gcd(f_m(x),g_m(x))=1$. Otherwise,
if we let $d(x):=\gcd(f_m(x),g_m(x))$, then
$e:=\deg d(x)\ge1$ and
$$d(x^2)=\gcd(f_m(x^2),g_m(x^2)).$$
Then from (\ref{eq2.3}) we can deduce that
$d(x^2)\mid (x-1)^m$. So we must have
$d(x^2)=(x-1)^l$ for some integer $1\le l\le m$.
Furthermore, we have
$$l=\deg d(x^2)=2\deg d(x)=2e.$$
Hence $d(x^2)=(x-1)^{2e}$ which is impossible
since the left-hand side contain no terms of odd
degree but the right-hand side contains the terms
of odd degree. So the claim is proved.

It is obvious that $|\det A_m|=|\det B_m|$.
So by Lemma 2.1, we have
$$|\det A_m|=|\det B_m|\not=0$$
as required. Lemma 2.2 is proved.
\end{proof}

\begin{lemma} \cite{[Er]} Let $m$ and $k$ be
nonnegative integers such that $k\le m$. Then
$$\upsilon_2(\binom{m}{k})=\sum_{i=1}^{\infty}
\Big(\Big\lfloor \frac{m}{2^i}\Big\rfloor
-\Big\lfloor\frac{k}{2^i}\Big\rfloor-
\Big\lfloor\frac{m-k}{2^i}\Big\rfloor\Big).$$
\end{lemma}

\begin{lemma} Let $m\ge 2$ be an integer. Then
$2^{\lfloor \frac{m}{2}\rfloor}|\det A_m$.
\end{lemma}
\begin{proof}
We divide the proof into the following two cases.

{\sc Case 1}. $m$ is even. Then $A_m=A_m^{(2)}$.
By Lemma 2.3, one knows that for any
$k\in\{0,1,...,\frac{m}{2}-1\}$, we obtain that
\begin{align}\label{eq2.3'}
\upsilon_2(\binom{m}{2k+1})
=\sum_{i=1}^{\lfloor \log m/\log 2\rfloor}
\Big(\Big\lfloor\frac{m}{2^i}\Big\rfloor
-\Big\lfloor\frac{2k+1}{2^i}\Big\rfloor
-\Big\lfloor \frac{m-2k-1}{2^i}\Big\rfloor\Big).
\end{align}

Since $m$ is even, one may let $m=2m_1$ with $m_1$
being an integer. Hence
\begin{align}\label{eq2.3''}
&\Big\lfloor\frac{m}{2}\Big\rfloor
-\Big\lfloor\frac{2k+1}{2}\Big\rfloor
-\Big\lfloor \frac{m-2k-1}{2}\Big\rfloor\notag\\
=& m_1-\Big\lfloor k+\frac{1}{2}\Big\rfloor
-(m_1-k+\Big\lfloor\frac{-1}{2}\Big\rfloor)=1.
\end{align}
Noticing that for any integer $i$ with $2\le i\le\lfloor\log m/\log 2\rfloor$,
\begin{align}\label{eq2.3'''}
\Big\lfloor\frac{m}{2^i}\Big\rfloor
-\Big\lfloor\frac{2k+1}{2^i}\Big\rfloor
-\Big\lfloor \frac{m-2k-1}{2^i}\Big\rfloor\ge 0,
\end{align}
it then follows from (2.4), (2.5) and (2.6) that for any
$k\in\{0,1,...,\frac{m}{2}-1\}$, one has
$$\upsilon_2(\binom{m}{2k+1})\ge 1.$$
Thus $2\mid\binom{m}{2k+1}$ for any $k\in\{0,1,...,\frac{m}{2}-1\}$.
In other words, for each $k\in\{0,1,...,\frac{m}{2}-1\}$,
all the elements in each $(2k+1)$-th row vector is
divisible by 2. This implies that $2^{\frac{m}{2}}|\det A_m$
as required. Lemma 2.4 is proved in this case.

{\sc Case 2}. $m$ is odd. Then $A_m=A_m^{(1)}$. For any
$k\in\{0,1,...,\frac{m-3}{2}\}$, first letting all the elements of
the $(2k+1)$-th row of $A_m$ multiply by $-1$, and then adding to
the corresponding terms of the $(2k+2)$-th row, we obtain a new
matrix, says $C_m$. Then $\det A_m=\det C_m$. Clearly, all the
nonzero elements of the $(2k+2)$-th row of $C_m$ are given by
$$\binom{m}{2k+1}+\binom{m}{2k}=\binom{m+1}{2k+1}$$
for $0\le k\le\frac{m-3}{2}$. For these terms, we compute their
2-adic valuation as follows:
\begin{align*}
\upsilon_2(\binom{m}{2k+1}+\binom{m}{2k})
=& \upsilon_2(\binom{m+1}{2k+1})\\
=& \sum_{i=1}^{\infty}\Big(\Big\lfloor\frac{m+1}{2^i}\Big\rfloor
-\Big\lfloor\frac{2k+1}{2^i}\Big\rfloor
-\Big\lfloor\frac{m-2k}{2^i}\Big\rfloor\Big).
\end{align*}

Since $m$ is odd, we let $m=2m_1+1$. Then $\lfloor\frac{m}{2}\rfloor=m_1$
and for any $k\in\{0,1,...,\frac{m-3}{2}\}$, we have
$$\Big\lfloor\frac{m+1}{2}\Big\rfloor-\Big\lfloor\frac{2k+1}{2}\Big\rfloor
-\Big\lfloor\frac{m-2k}{2}\Big\rfloor
=(m_1+1)-k-\Big\lfloor m_1-k+\frac{1}{2}\Big\rfloor=1.$$
From the above identities, we can derive that
$$\upsilon_2(\binom{m+1}{2k+1})\ge 1$$
for $0\le k\le\frac{m-3}{2}=m_1-1$. It follows that for each
$0\le k\le m_1-1$, all the nonzero elements (and hence all the
elements) of the $(2k+2)$-th row of $C_m$ are divisible by 2. It follows that
$$2^{\frac{m-1}{2}}=2^{m_1}\mid \det C_m.$$
Namely, one has the desired result $2^{\lfloor\frac{m}{2}\rfloor}|\det A_m$.
So Lemma 2.4 is proved in this case.
\end{proof}

\begin{lemma}
Let $m\ge 3$ be an integer. Then there exists an eigenvalue
of $A_m$, denoted by $\lambda_1$, such that $|\lambda_1|>1$.
\end{lemma}
\begin{proof} First of all, the order of $A_m$ is $m-1\ge 2$ since
$m\ge 3$. Let $\lambda_1,...,\lambda_{m-1}$ be the $m-1$ eigenvalues
of $A_m$. Then $\det A_m=\prod_{i=1}^{m-1}\lambda_i$. Then by Lemmas
2.2 and 2.4, we deduce that
$$\prod_{i=1}^{m-1}|\lambda_i|=\Big|\prod_{i=1}^{m-1}\lambda_i\Big|
=|\det A_m|\ge 2^{\lfloor \frac{m}{2}\rfloor}\ge 2>1.$$
Hence there is at least one of $\lambda_1,...,\lambda_{m-1}$,
says $\lambda_1$, such that $|\lambda_1|>1$.

So Lemma 2.5 is proved.
\end{proof}

By Lemma 2.5 and the Jordan canonical form, we may let
$$A_m=P^{-1}K_mP,$$
where $P$ is a nonsingular matrix of order $m-1$ and
$$K_m:={\rm diag}(J_{\lambda_1}, A_{1,m})$$
with
$$J_{\lambda_1}=\lambda_1 I_{r}+N_{r} (1\le r\le m-1)$$
is the first Jordan block, $I_{r}$ is the identity matrix of order $r$
and $N_{r}$ is the nilpotent matrix satisfying $N_r^r=O_{r}$
with $O_r$ being the zero matrix of order $r$.

\begin{lemma} \cite{[MG]}
Let $m$ be a positive integer. Then $t_m(0)=1,t_m(1)=-m$
and
$$t_m(2n)=\sum_{j=0}^{\lfloor\frac{m}{2}\rfloor}
\Big(^m_{2j}\Big) t_m(n-j),$$
$$t_m(2n+1)=-\sum_{j=0}^{\lfloor\frac{m-1}{2}
\rfloor}\Big(^{\;\: m}_{2j+1}\Big) t_m(n-j),$$
where $t_m(n):=0$ if $n<0$.
\end{lemma}

\begin{lemma}
For any positive integer $d$, we have
$$\left(
\begin{array}{c}
t_m(2^dn-1)\\
t_m(2^dn-2)\\
.\\
.\\
.\\
t_m(2^dn-m+1)\\
\end{array}
\right)=A_m^d
\left(\begin{array}{c}
t_m(n-1)\\
t_m(n-2)\\
.\\
.\\
.\\
t_m(n-m+1)\\
\end{array}
\right).
$$
\end{lemma}

\begin{proof}
By Lemma 2.6, we have
\begin{align}
t_m(2n-1)=& -\sum_{j=0}^{\lfloor\frac{m-1}{2}\rfloor}
\binom{m}{2j+1}t_m(n-1-j)\nonumber\\
=& (-\binom{m}{1},-\binom{m}{3},...,-\binom{m}
{2\lfloor\frac{m-1}{2}\rfloor})_{1\times (\lfloor\frac{m-1}{2}\rfloor+1)}
\cdot \left(\begin{array}{c}
t_m(n-1)\\
t_m(n-2)\\
.\\
.\\
.\\
t_m(n-\lfloor\frac{m-1}{2}\rfloor-1)
\end{array}
\right)_{(\lfloor\frac{m-1}{2}\rfloor+1)\times 1}\nonumber\\
=& (-\binom{m}{1},-\binom{m}{3},...,
-\binom{m}{2\lfloor\frac{m-1}{2}\rfloor},0...,0)_{1\times (m-1)}\cdot
\left(\begin{array}{c}
t_m(n-1)\\
t_m(n-2)\\
.\\
.\\
.\\
t_m(n-\lfloor\frac{m-1}{2}\rfloor-1)\\
t_m(n-\lfloor\frac{m-1}{2}\rfloor-2)\\
.\\
.\\
.\\
t_m(n-m+1)
\end{array}
\right)_{(m-1)\times 1}\nonumber\\
\end{align}
and
\begin{align}
t_m(2n-2)=& \sum_{j=0}^{\lfloor\frac{m}{2}\rfloor}
\binom{m}{2j} t_m(n-1-j)\nonumber\\
=& (\binom{m}{0},\binom{m}{2},...,\binom{m}
{2\lfloor\frac{m}{2}\rfloor})_{1\times (\lfloor\frac{m}{2}\rfloor+1)}
\cdot \left(\begin{array}{c}
t_m(n-1)\\
t_m(n-2)\\
.\\
.\\
.\\
t_m(n-\lfloor\frac{m}{2}\rfloor-1)
\end{array}
\right)_{(\lfloor\frac{m}{2}\rfloor+1)\times 1}\nonumber\\
=& (\binom{m}{0},\binom{m}{2},...,
\binom{m}{2\lfloor\frac{m}{2}\rfloor},0...,0)_{1\times (m-1)}\cdot
\left(\begin{array}{c}
t_m(n-1)\\
t_m(n-2)\\
.\\
.\\
.\\
t_m(n-\lfloor\frac{m}{2}\rfloor-1)\\
t_m(n-\lfloor\frac{m}{2}\rfloor-2)\\
.\\
.\\
.\\
t_m(n-m+1)
\end{array}
\right)_{(m-1)\times 1}.\nonumber\\
\end{align}
Similarly, we have
\begin{align}
t_m(2n-3)=& -\sum_{j=0}^{\lfloor\frac{m-1}{2}\rfloor}
\binom{m}{2j+1}t_m(n-2-j)\nonumber\\
=& (-\binom{m}{1},-\binom{m}{3},...,-\binom{m}
{2\lfloor\frac{m-1}{2}\rfloor})_{1\times (\lfloor\frac{m-1}{2}\rfloor+1)}
\cdot \left(\begin{array}{c}
t_m(n-2)\\
t_m(n-3)\\
.\\
.\\
.\\
t_m(n-\lfloor\frac{m-1}{2}\rfloor-2)
\end{array}
\right)_{(\lfloor\frac{m-1}{2}\rfloor+1)\times 1}\nonumber\\
=& (0,-\binom{m}{1},-\binom{m}{3},...,
-\binom{m}{2\lfloor\frac{m-1}{2}\rfloor},0...,0)_{1\times (m-1)}\cdot
\left(\begin{array}{c}
t_m(n-1)\\
t_m(n-2)\\
.\\
.\\
.\\
t_m(n-\lfloor\frac{m-1}{2}\rfloor-2)\\
t_m(n-\lfloor\frac{m-1}{2}\rfloor-3)\\
.\\
.\\
.\\
t_m(n-m+1)
\end{array}
\right)_{(m-1)\times 1} \nonumber\\
\end{align}
and
\begin{align}
t_m(2n-4)=& \sum_{j=0}^{\lfloor\frac{m}{2}\rfloor}
\binom{m}{2j} t_m(n-1-j)\nonumber\\
=& (\binom{m}{0},\binom{m}{2},...,\binom{m}
{2\lfloor\frac{m}{2}\rfloor})_{1\times (\lfloor\frac{m}{2}\rfloor+1)}
\cdot \left(\begin{array}{c}
t_m(n-2)\\
t_m(n-3)\\
.\\
.\\
.\\
t_m(n-\lfloor\frac{m}{2}\rfloor-2)
\end{array}
\right)_{(\lfloor\frac{m}{2}\rfloor+1)\times 1}\nonumber\\
=& (0,\binom{m}{0},\binom{m}{2},...,
\binom{m}{2\lfloor\frac{m}{2}\rfloor},0...,0)_{1\times (m-1)}\cdot
\left(\begin{array}{c}
t_m(n-1)\\
t_m(n-2)\\
.\\
.\\
.\\
t_m(n-\lfloor\frac{m}{2}\rfloor-2)\\
t_m(n-\lfloor\frac{m}{2}\rfloor-3)\\
.\\
.\\
.\\
t_m(n-m+1)
\end{array}
\right)_{(m-1)\times 1}.\nonumber\\
\end{align}

Now we consider the linear relation between
$t_m(2n-1),t_m(2n-2),...,t_m(2n-m+1)$ and
$t_m(n-1),t_m(n-2),...,t_m(n-m+1)$. We have
\begin{align}\label{eq2.4}
\left(
\begin{array}{c}
t_m(2n-1)\\
t_m(2n-2)\\
.\\
.\\
.\\
t_m(2n-m+1)\\
\end{array}
\right)=A_m
\left(
\begin{array}{c}
t_m(n-1)\\
t_m(n-2)\\
.\\
.\\
.\\
t_m(n-m+1)\\
\end{array}
\right).
\end{align}
It then follows that
$$\left(
\begin{array}{c}
t_m(2^dn-1)\\
t_m(2^dn-2)\\
.\\
.\\
.\\
t_m(2^dn-m+1)\\
\end{array}
\right)=A_m
\left(
\begin{array}{c}
t_m(2^{d-1}n-1)\\
t_m(2^{d-1}n-2)\\
.\\
.\\
.\\
t_m(2^{d-1}n-m+1)\\
\end{array}
\right)
=...
=A_m^d
\left(
\begin{array}{c}
t_m(n-1)\\
t_m(n-2)\\
.\\
.\\
.\\
t_m(n-m+1)\\
\end{array}
\right)
$$
as desired.

This finishes the proof of Lemma 2.7.
\end{proof}

Since $A_m=P^{-1}K_mP$, we have $PA_m^d=K_m^d P$. It follows from
Lemma 2.7 and $K_m:={\rm diag}\{J_{\lambda_1}, A_{1,m}\}$ that
\begin{align}
P\left(
\begin{array}{c}
t_m(2^dn-1)\\
t_m(2^dn-2)\\
.\\
.\\
.\\
t_m(2^dn-m+1)\notag\\
\end{array}
\right)
=& K_m^dP
\left(
\begin{array}{c}
t_m(n-1)\\
t_m(n-2)\\
.\\
.\\
.\\
t_m(n-m+1)\\
\end{array}
\right)\\
=&{\rm diag}(J^d_{\lambda_1},A_{1,m}^d)P
\left(
\begin{array}{c}
t_m(n-1)\\
t_m(n-2)\\
.\\
.\\
.\\
t_m(n-m+1)\\
\end{array}
\right).
\end{align}

Let
\begin{equation}\label{eq1}
\overrightarrow{b_n}:
=P\left(\begin{array}{c}
t_m(n-1)\\
t_m(n-2)\\
.\\
.\\
.\\
t_m(n-m+1)\\
\end{array}\right).
\end{equation}
Then
\begin{equation}\label{eq2}
\overrightarrow{b_{2^dn}}={\rm diag}(J^d_{\lambda_1}, A_{1,m}^d) \overrightarrow{b_n}.
\end{equation}

In what follows, we define $(\overrightarrow{b_n})_i$ to be the
$i$-th component of $\overrightarrow{b_n}$. In other words,
$$\overrightarrow{b_n}=((\overrightarrow{b_n})_1,
(\overrightarrow{b_n})_2,...,(\overrightarrow{b_n})_{m-1})^T.$$

Since ${\rm rank}(P)=m-1$, we have
$\overrightarrow{b_n}\not=\overrightarrow{0}
\iff P^{-1}\overrightarrow{b_n}\not=\overrightarrow{0}$.

\begin{definition}
We say that the sequence $\{a_n\}_{n=1}^\infty$ of complex numbers is
{\it periodic} if there exists a positive integer $l$ such that
$a_n=a_{n+l}$ for any positive integer $n$.
\end{definition}

\begin{definition}
We say that $a$ is a {\it limit point} of the sequence $\{a_n\}_{n=1}^\infty$
if there is a convergent subsequence $\{b_n\}_{n=1}^\infty$ of
$\{a_n\}_{n=1}^\infty$ such that $\lim_{n\rightarrow\infty}b_n=a$.
\end{definition}

For the simplicity of the proofs given below, we use the symbol $\Delta$
to denote the shift operator which is defined in the following.

\begin{definition}
The {\it shift operator} $\Delta$ is defined inductively to
act on the sequence $\{b_n\}_{n=1}^\infty$ by
$\Delta(b_n):=b_{n+1}$ for any nonnegative integer $n$.

For any $i\in\mathbb{N}_+$, let $\Delta^0(b_n):=b_n$ (we use $I$ to
denote the symbol $\Delta^0$, the {\it identity operator}), and
$\Delta^i(b_n):=\Delta(\Delta^{i-1}(b_n))$
when $i\ge 1$. Then
$$\Delta^i(b_n)=b_{n+i} \ \forall i\ge 0.$$
Furthermore, for any given polynomial $g(x)=\sum_{i=0}^kc_ix^i$,
we define the polynomial operator $g(\Delta)$ as follows:
$$g(\Delta)(b_n):=\sum_{i=0}^kc_i\Delta^i(b_n).$$
Then
\begin{align}\label{2.13}
g(\Delta)(b_n)=\sum_{i=0}^kc_i b_{n+i}.
\end{align}
\end{definition}

The following result is known and is due to Bolzano and Weierstrass, see for example,
\cite[Theorem 3.24] {[Apostol]}.

\begin{lemma} \cite[Theorem 3.24] {[Apostol]} \label{lemma2.8}
If a bounded set $S$ in $\mathbf{R}^n$ contains infinitely many points, then
there is at least one point in $\mathbf{R}^n$ which is an accumulation point of $S$.
\end{lemma}

\begin{lemma} \label{lemma2.9}
Let $c\in\mathbb{C}$ and let $\{a_n\}_{n=1}^{\infty}$ be a bounded sequence of
complex numbers with only finitely many limit points. Define
\begin{align*}
b_n:=(\Delta-cI)(a_n)=a_{n+1}-ca_n.
\end{align*}
Then the sequence \(\{b_n\}_{n=1}^{\infty}\) has only finitely many limit points.
\end{lemma}

\begin{proof}
Let $A$ denote the set of all limit points of $\{a_n\}_{n=1}^{\infty}$. By hypothesis, $A$ is a finite set.
Consider the set
$$
B := \{u-cv\mid u, v \in A\}.
$$
Since $A$ is finite, $B$ is also finite.

Now we prove that every limit point of $\{b_n\}_{n=1}^\infty$ belongs to $B$. Pick an arbitrary
limit point $w$ of $\{b_n\}_{n=1}^\infty$. Then there exists a strictly increasing sequence of
positive integers $\{k_n\}_{n=1}^{\infty}$ such that
\begin{align}\label{eq2.14}
\lim_{n\to\infty}(a_{k_n+1}-ca_{k_n})=\lim_{n\to \infty}b_{k_n}=w.
\end{align}
Since $\{a_n\}_{n=1}^\infty$ is bounded, the subsequence $\{a_{k_n}\}_{n=1}^{\infty}$ is bounded.
By the Bolzano-Weierstrass theorem {\rm {\cite[Theorem 3.24]{[Apostol]}}}, $\{a_{k_n}\}_{n=1}^{\infty}$
has a convergent subsequence, denoted by $\{a_{k_n'}\}_{n=1}^{\infty}$, where $\{k_n'\}_{n=1}^{\infty}$
is a subsequence of $\{k_n\}_{n=1}^\infty$, and write
\begin{align}\label{eq2.15}
\lim_{n \to \infty} a_{k_n'}:=p\in A.
\end{align}

Since the original subsequence $\{b_{k_n}\}_{n=1}^\infty$ converges to $w$, the subsequence $\{b_{k_n'}\}_{n=1}^\infty$
also converges to $w$. Since $\{a_n\}_{n=1}^\infty$ is bounded, so does the subsequence $\{a_{k_n'+1}\}_{n=1}^{\infty}$.
Applying the Bolzano-Weierstrass theorem again, the subsequence $\{a_{k_n'+1}\}_{n=1}^{\infty}$ has a convergent
subsequence, denoted it by $\{a_{k_n^{''}+1}\}_{n=1}^{\infty}$, where $\{k_n^{''}\}_{n=1}^{\infty}$
is a subsequence of $\{k_n'\}_{n=1}^\infty$, and write
\begin{align}\label{eq2.16}
\lim_{n\to \infty} a_{k_n^{''}+1}:=q\in A.
\end{align}
Notice that by \eqref{eq2.15}, one has
\begin{align}\label{eq2.17}
\lim_{n \to \infty} a_{k_n^{''}}=p\in A.
\end{align}

But $\{k_n^{''}\}_{n=1}^{\infty}$ is a subsequence of $\{k_n\}_{n=1}^\infty$. Then putting \eqref{eq2.14},
\eqref{eq2.16} and \eqref{eq2.17} together gives us that
\begin{align*}
w=\lim_{n\to \infty}(a_{k_n^{''}+1}-ca_{k_n^{''}})=q-cp\in B.
\end{align*}
In other words, every limit point of $\{b_n\}_{n=1}^\infty$ belongs to the finite set $B$.
Therefore the sequence $\{b_n\}_{n=1}^\infty$ has only finitely many limit points.

This finishes the proof of Lemma \ref{lemma2.9}.
\end{proof}

\begin{lemma}\label{lemma2.10}
Let $q_1$ and $\mu$ be complex numbers with $q_1\ne 0$ and for any
integer $n\ge 1$, let $q_n:=q_1\mu^{n-1}$ be such that
$\underset{n\rightarrow\infty}{\lim}|q_n|\ne \infty$. Then the sequence
$\{q_n\}_{n=1}^\infty$ has a finite number of limit point if and
only if either $|\mu|<1$, or there is an integer $m\ge 1$ such that
$\mu^m=1$.
\end{lemma}

\begin{proof}
First of all, we show the necessity part. Let the sequence
$\{q_n\}_{n=1}^\infty$ have a finite number of limit point.
Suppose that $|\mu |>1$. Since $q_1\ne 0$, we deduce that
$|q_n|=|q_1| |\mu |^{n-1}\rightarrow\infty$ as $n\rightarrow\infty$.
This contradicts with the hypothesis that $\underset{n\rightarrow\infty}
{\lim}|q_n|\ne \infty$. Hence we have either $|\mu|<1$ or $|\mu|=1$.
Let now $|\mu|=1$. Then we can write $\mu =e^{i\theta}$ with
$\theta\in (0, 2\pi]$.

If $\frac{\theta}{2\pi}\not\in\mathbb{Q}$, then by \cite{[WE]}, one knows
that the sequence $\{\frac{m\theta}{2\pi}-\lfloor\frac{m\theta}{2\pi}\rfloor\}_{m=1}^\infty$
is dense in the interval $(0,1)$. This implies that the set of limit points
of the sequence $\{q_n\}$ is equal to the interval $(0,1)$. It follows that
the sequence $\{q_n\}_{n=1}^\infty$ has infinitely many limit points.
This is a contradiction. Thus we must have $\frac{\theta}{2\pi}\in\mathbb{Q}_{>0}$.
Let $\frac{\theta}{2\pi}=\frac{s}{t}$ with $s$ and $t$ beng positive integers.
Therefore
$$\mu^t=e^{i\cdot 2\pi\cdot \frac{s}{t}\cdot t}=e^{2\pi i s}=1.$$
So the necessity part is proved.

Next, we show the sufficiency part. Its proof is divided into the following two cases.

{\sc Case 1.} $|\mu |<1$. Then as $n\rightarrow\infty$, one has $n-1\rightarrow\infty$,
and so $|q_n|=|q_1| |\mu|^{n-1}\rightarrow 0$. This infers that $q_n\rightarrow 0$ as
$n\rightarrow\infty$. Hence 0 is the only limit point of the sequence $\{q_n\}_{n=1}^\infty$.
Hence the sufficiency part is true in this case.

{\sc Case 2.} $\mu^m=1$ for some integer $m\ge 1$. Then for any integer
$n\ge 2$, one has $n-1=ml+r$, where $l$ and $r$ are nonnegative integers
such that $l\ge 0$ and $0\le r\le m-1$. So
$$
q_n=q_1\mu^{n-1}=q_1\mu^{ml+r}=q_1(\mu^m)^l\mu^r=q_1\mu^r.
$$
It then follows that
$$
q_n=q_1\mu^r\in\{q_1, q_1\mu, q_1\mu^2,...,q_1\mu^{m-1}\}.
$$
That is, the set $\{q_1, q_1\mu, q_1\mu^2,...,q_1\mu^{m-1}\}$ equals
the set of all the limit points of the sequence $\{q_n\}_{n=1}^\infty$.
The sufficiency part is proved in this case.

The proof of Lemma \ref{lemma2.10} is complete.
\end{proof}

\begin{lemma}\label{lemma2.11'}\cite[Corollary 2.24]{[El]}
Let $\{a_n\}_{n=1}^\infty$ be a sequence of complex numbers. Let $\mu_1,...,\mu_r$
be $r\ge 1$ distinct complex numbers and let $m_1,...,m_r$ be $r$ positive integers.
Then for any integer $n\ge 1$, the general solution of the equation
\begin{align*}
(\Delta-\mu_1 I)^{m_1}(\Delta-\mu_2 I)^{m_2}\ldots
(\Delta-\mu_rI)^{m_r}(a_n)=0
\end{align*}
is given by
$$a_n=\sum_{i=1}^r(a_{i,0}+a_{i,1}n+a_{i,2}n^2+\ldots +
a_{i,m_{i}-1}n^{m_i-1})\mu_i^{n}$$
where $a_{i,j}\in\mathbb{C}$ for all $1\le i\le r$ and $1\le j\le m_i-1$.
\end{lemma}

\begin{lemma}\label{lemma2.11}
Let $\{a_n\}_{n=1}^\infty$ be a bounded sequence with finitely
many limit points. If there is an integer $k\ge 1$ and there are
$k+1$ complex numbers $c_0,c_1,...,c_k\in\mathbb{C}$
with $c_0c_k\ne 0$ such that $\sum_{i=0}^kc_ia_{n+i}=0$
for any positive integer $n$, then $a_n=\alpha_n+\beta_n$,
where $\{\alpha_n\}_{n=1}^\infty$ is a periodic sequence and
$\lim_{n\rightarrow\infty}|\beta_n|=0$.
\end{lemma}

\begin{proof}
Let $g(x):=\sum_{i=0}^kc_ix^i$. Then by (\ref{2.13}), one has
$$g(\Delta)(a_n)=\sum_{i=0}^kc_ia_{n+i}.$$

Let $\{\mu_1,...,\mu_r\}$ be the all distinct roots of
$g(x)$ with multiplicities $m_1,\ldots,m_r$, respectively. Since
$c_0c_k\not=0$, we have $\mu_1^{m_1}\ldots\mu_r^{m_r}\ne 0$.
Without loss of generality, we may assume that $c_k=1$. Then
$$g(x)=\prod_{i=1}^r(x-\mu_i)^{m_i}$$
which implies that
$$g(\Delta)=\prod_{i=1}^r(\Delta-\mu_iI)^{m_i}.$$
For any integer $i$ with $1\le i\le r$, one defines
\begin{align}\label{eq2.18}
q_{n,i}:=\frac{g(\Delta)}{\Delta-\mu_iI}(a_n)
=(\Delta-\mu_iI)^{m_i-1}\prod_{j=1 \atop j\neq i}^{r}(\Delta-\mu_j I)^{m_j}(a_n).
\end{align}
Then
\begin{align*}
(\Delta-\mu_iI)q_{n,i}=(\Delta-\mu_iI)\frac{g(\Delta)}{\Delta-\mu_iI}(a_n)
=g(\Delta)(a_n)=\sum_{i=0}^{k}c_ia_{n+i}=0.
\end{align*}
This infers that
$$\Delta(q_{n,i})=\mu_iI(q_{n,i})=\mu_i q_{n,i}.$$
But $\Delta(q_{n,i})=q_{n+1,i}$.
Thus for any integer $n\ge 1$, one has $q_{n+1,i}=\mu_i q_{n,i}$.
In other words,
$$q_{n,i}=q_{1,i}\mu_i^{n-1}$$
for any integers $n$ and $i$ with $n\ge 1$ and $1\le i\le r$.

Since $\{a_n\}_{n=1}^\infty$ has finitely many limit points, by \eqref{eq2.18},
applying Lemma \ref{lemma2.9} for $k-1$ times yields that $\{q_{n, i}\}_{n=1}^{\infty}$
also has finitely many limit points. It then follows from Lemma \ref{lemma2.10} that
for any integer $i$ with $1\le i\le r$, we have either $|\mu_i|<1$ or there is an integer
$\ell_i\ge 1$ such that $\mu_i$ is an $\ell_i$-th primitive root of unity.
Moreover, the equation
$$g(\Delta)(a_n)=\prod_{i=1}^{r}(\Delta-\mu_i I)^{m_i}(a_n)=0$$
is a linear homogeneous difference equation with the constant coefficients $c_0, c_1, ..., c_k$.
Then by Lemma \ref{lemma2.11'}, we have
$$a_n=\sum_{i=1}^rd_i(n)\mu_i^n$$
where for all $1\le i\le r$, one has
$$d_i(x):=a_{i,0}+a_{i,1}x+a_{i,2}x^2+\ldots+a_{i,m_{i}-1}x^{m_i-1}\in\mathbb{C}[x].$$

Without loss of generality, one may assume that $\mu_i$ is an $\ell_i$-th primitive root
of unity for $1\le i\le t$ and $|\mu_i|<1$ for $t+1\le i\le r$, where $0\le t\le r$.
For any integer $n\ge 1$, let
$$\alpha_n:=\sum_{i=1}^td_i(n)\mu_i^n$$
and
$$\beta_n:=\sum_{i=t+1}^r d_i(n)\mu_i^n.$$
Then $a_n=\alpha_n+\beta_n$.

First of all, we show that $\beta_n\rightarrow 0$ as $n\rightarrow\infty$. Actually,
since $|\mu_i|<1$ for all $t+1\le i\le r$, we have
\begin{align*}
|\beta_n|=& |\sum_{i=t+1}^r d_i(n)\mu_i^n|\\
\le & \sum_{i=t+1}^r |d_i(n)| |\mu_i|^n\\
= & \sum_{i=t+1}^r |\sum_{l=0}^{m_i-1} a_{i,l} n^l| |\mu_i|^n\\
\le & \sum_{i=t+1}^r \sum_{l=0}^{m_i-1} |a_{i,l}|\cdot n^l\cdot |\mu_i|^n\\
= & \sum_{i=t+1}^r \sum_{l=0}^{m_i-1} |a_{i,l}|\cdot \frac{n^l}{|\frac{1}{\mu_i}|^n}\\
\rightarrow & \sum_{i=t+1}^r \sum_{l=0}^{m_i-1} |a_{i,l}|\cdot 0=0 \ {\rm as} \ n\rightarrow\infty,
\end{align*}
where in the last step, we used the basic fact that $n^l/|\frac{1}{\mu_i}|^n\rightarrow 0$
as $n\rightarrow\infty$. Hence
$$
\lim_{n\rightarrow\infty}|\beta_n|=0
$$
as desired.

Second, we have
\begin{align*}
\alpha_n=&\sum_{i=1}^{t}d_i(n)\mu_i^n
=\sum_{i=1}^{t}\Big(\sum_{j=0}^{m_i-1}a_{i,j}n^j\Big)\mu_i^n \nonumber\\
=&\sum_{i=1}^{t}a_{i,0}\mu_i^n+\sum_{i=1}^{t}
\Big(\sum_{j=1}^{m_i-1}a_{i,j}n^j\Big)\mu_i^n \nonumber\\
=&\sum_{i=1}^{t}a_{i,0}\mu_i^n+\sum_{j=1}^{M-1}
\Big(\sum_{i=1}^{t}a_{i,j}\mu_i^n\Big)n^j,
\end{align*}
where $M:=\max(m_1,\ldots,m_t)$ and $a_{i, j}:=0$ if $j>m_i-1$. For $0\le j\le M-1$, let
$$
c_j(n):=\sum_{i=1}^{t}a_{i,j}\mu_i^n.
$$
Then
\begin{align}\label{eq2.19}
\alpha_n=c_0(n)+\sum_{j=1}^{M-1}c_j(n)n^j.
\end{align}

Let $\ell:={\rm lcm}(\ell_1,\ldots,\ell_m)$. Since $\mu_i$ is an $\ell_i$-th primitive
root of unity, we have
\begin{align}\label{eq2.20}
c_j(n)=c_j(\langle n\rangle_\ell),
\end{align}
where $0\le \langle n\rangle_\ell\le \ell-1$ and $\langle n\rangle_\ell\equiv n\pmod \ell$.

{\sc Claim} that $c_j(n)=0$ for all $1\le j\le M-1$ and for any integer $1\le n\le \ell$.
Otherwise, there exist integers $T$ and $n_0$ with $1\le T \le M-1$ and $1\le n_0\le \ell$
such that $c_T(n_0)\neq 0$ and $c_{T+1}(n_0)=\ldots=c_{M-1}(n_0)=0$. For any integer $u\ge 1$, we have
\begin{align*}
\alpha_{n_0+\ell u}=&c_0({\langle n_0+\ell u\rangle_\ell})
+c_T(\langle n_0+\ell u\rangle_\ell)(n_0+\ell u)^T
+c_{T-1}(\langle n_0+\ell u\rangle_\ell)(n_0+\ell u)^{T-1} \nonumber\\
&+\ldots+c_1(\langle n_0+\ell u\rangle_\ell)(n_0+\ell u) \nonumber\\
=&c_0(n_0)+c_T(n_0)(n_0+\ell u)^T+c_{T-1}(n_0)(n_0+\ell u)^{T-1}
+\ldots+c_1(n_0)(n_0+\ell u) \nonumber\\
=&c_0(n_0)+(n_0+\ell u)^T\Big(c_T(n_0)+\frac{c_{T-1}(n_0)}{n_0+\ell u}
+\ldots+\frac{c_1(n_0)}{(n_0+\ell u)^{T-1}} \Big).
\end{align*}
It then follows that
\begin{align*}
|\alpha_{n_0+\ell u}|=&\Big|c_0(n_0)+(n_0+\ell u)^T
\Big(c_T(n_0)+\frac{c_{T-1}(n_0)}{n_0+\ell u}
+\ldots+\frac{c_1(n_0)}{(n_0+\ell u)^{T-1}} \Big)\Big| \nonumber\\
\ge & \Big|(n_0+\ell u)^T\Big(c_T(n_0)+\frac{c_{T-1}(n_0)}{n_0+\ell u}
+\ldots+\frac{c_1(n_0)}{(n_0+\ell u)^{T-1}} \Big)\Big|-|c_0(n_0)|
\nonumber\\
\ge &|n_0+\ell u|^T \Big(|c_T(n_0)|-\Big|\frac{c_{T-1}(n_0)}{n_0+\ell u}\Big|
-\ldots-\Big|\frac{c_1(n_0)}{(n_0+\ell u)^{T-1}}\Big|\Big)-|c_0(n_0)|
\end{align*}
Noticing that
$$\Big|\frac{c_j(n_0)}{(n_0+\ell u)^{T-j}}\Big|\rightarrow 0$$
as $u \rightarrow \infty$  for $1\le j\le T-1$ and $c_T(n_0)\neq 0$, we have
$$
|n_0+\ell u|^T \Big(|c_T(n_0)|-\Big|\frac{c_{T-1}(n_0)}{n_0+\ell u}\Big|
-\ldots-\Big|\frac{c_1(n_0)}{(n_0+\ell u)^{T-1}}\Big|\Big)-|c_0(n_0)|\rightarrow \infty
$$
as $u\rightarrow \infty$. Therefore one derives that $|\alpha_{n_0+\ell u}|\rightarrow\infty$
as $u\rightarrow \infty$. But
$$
|a_{n_0+\ell u}|=|\alpha_{n_0+\ell u}+\beta_{n_0+\ell u}|
\ge |\alpha_{n_0+\ell u}|-|\beta_{n_0+\ell u}|.
$$
It follows from the fact $\lim_{u\rightarrow\infty}|\beta_{n_0+\ell u}|=0$
that $|a_{n_0+\ell u}|\rightarrow \infty$ as $u\rightarrow \infty$. This contradicts
with the hypothesis that $\{a_n\}_{n=1}^\infty$ is bounded. Hence $c_j(n)=0$ for all
$1\le j\le M-1$ and for any integer $1\le n\le \ell$. The claim is proved.

Finally, for any positive integer $n$, by the claim and \eqref{eq2.19} as well as \eqref{eq2.20}, we have
$$
\alpha_{n+\ell}=c_0(n+\ell)=c_0(\langle n+\ell\rangle_\ell)=c_0(\langle n\rangle_\ell)=c_0(n)=\alpha_n.
$$
Hence $\{\alpha_n\}_{n=0}^\infty$ is a periodic sequence with $\ell$ as its period.

This completes the proof of Lemma \ref{lemma2.11}.
\end{proof}

\section{Proof of Theorem 1.1}

In this section, we present the proof of Theorem 1.1. We begin with the
following known lemma.

\begin{lemma} (\textsection 2 of Chapter 16 of \cite{[VA]}, Weierstrass M-test) \label{lem5}
Let the series $\sum_{n=0}^\infty a_n(x)$ and $\sum_{n=0}^\infty b_n(x)$
satisfy that $|a_n(x)|\le b_n(x)$ for every $x\in [0,1)$ and for all
sufficiently large $n\in\mathbb{N}$. If the series $\sum_{n=0}^\infty b_n(x)$
uniformly converges on $[0,1)$, then the series $\sum_{n=0}^\infty a_n(x)$
converges absolutely and uniformly on $[0,1)$.
\end{lemma}

\begin{lemma}\cite{[VA]} (\textsection 2 of Chapter 16 of \cite{[VA]}, Weierstrass M-test) \label{lem6}
If the power series $\sum_{n=0}^\infty c_nx^n$ converges at the point $0<x_0<1$,
then it converges uniformly on the interval $[0, x_0]$, and the sum of the
series is continuous on this interval.
\end{lemma}

\begin{lemma}\label{lem7}
The power series
$$\prod_{i=0}^\infty(1-x^{2^i})^m=\sum_{n=0}^{\infty}t_m(n)x^n$$
is continuous on the interval $[0,1)$.
\end{lemma}
\begin{proof}
By \cite{[MG]}, we have
\begin{align*}
\prod_{i=0}^\infty(1-x^{2^i})^m=&\Big(\sum_{n=0}^\infty(-1)^{s_2(n)}x^n\Big)^m\\
=& \Big(\sum_{n_1=0}^\infty(-1)^{s_2(n_1)}x^{n_1}\Big)...\Big(\sum_{n_m=0}^\infty(-1)^{s_2(n_m)}x^{n_m}\Big)\\
=&\sum_{n=0}^\infty\Big(\sum_{n_1+...+n_m=n}\prod_{i=1}^m(-1)^{s_2(n_i)}\Big) x^n\\
=&\sum_{n=0}^\infty\Big(\sum_{n_1+...+n_m=n}(-1)^{s_2(n_1)+...+s_2(n_m)}\Big) x^n,
\end{align*}
where $s_2(n)$ denotes the sum of (binary) digits function of $n$. It follows that
\begin{align*}
|t_m(n)|=& \Big|\sum_{n_1+...+n_m=n}(-1)^{s_2(n_1)+...+s_2(n_m)}\Big|\\
\le& \sum_{n_1+...+n_m=n}1\\
=& \binom{m+n-1}{n}.
\end{align*}
However,
$$\sum_{n=0}^\infty \binom{m+n-1}{n}x^n=\frac{1}{(1-x)^m},\quad 0\le |x|<1.$$
Hence by Lemma \ref{lem5}, one derives that the infinite product $\prod_{i=0}^\infty(1-x^{2^i})^m$
converges absolutely and uniformly on $[0,1)$. Then for every $x_0\in[0,1)$,
$\prod_{i=0}^\infty(1-x_0^{2^i})^m$ converges. Thus by Lemma \ref{lem6}, one knows that
the infinite product function $\prod_{i=0}^\infty(1-x^{2^i})^m$ is continuous on $[0,1)$.

This concludes the proof of Lemma 3.3.
\end{proof}

\begin{lemma} (\textsection Linear mapping of Chapter 1 of \cite{[Ru]}) \label{lem4}
Suppose $\Lambda$ is a linear functional on a topological vector space $R^{m-1}$, and there exists
$x \in R^{m-1}$ such that $\Lambda x \neq 0$. Then the following four properties are equivalent:

{\rm (i).} $\Lambda$ is continuous at $0$.

{\rm (ii).} The null space $N(\Lambda)$ is closed.

{\rm (iii).} $N(\Lambda)$ is not dense in $R^{m-1}$.

{\rm (iv).} $\Lambda$ is bounded in some neighborhood of $0$.
\end{lemma}

Finally, we are in the position to prove Theorem 1.1.\\
\\
{\it Proof of Theorem 1.1.}
We prove Theorem 1.1 by using contradiction. Suppose that the sequence
$\{t_m(n)\}_{n=0}^\infty$ is bounded. Let $\overrightarrow{b_n}$ be
defined as in (\ref{eq1}), i.e.,
\begin{equation*}
\overrightarrow{b_n}
=P\left(\begin{array}{c}
t_m(n-1)\\
t_m(n-2)\\
.\\
.\\
.\\
t_m(n-m+1)\\
\end{array}\right).
\end{equation*}
By Lemma \ref{lem4}, for $\overrightarrow{x}\in R^{m-1}$,
$f(\overrightarrow{x})=P\overrightarrow{x}$ is a continuous function.
Moreover, we have
$\{t_m(n)\}_{n=0}^\infty$ is bounded $\iff$ $\{\overrightarrow{b_n}\}_{n=0}^\infty$
is bounded $\iff$ each component of $\{\overrightarrow{b_n}\}_{n=0}^\infty$ is
bounded.

By (\ref{eq2}), and the paragraph exactly after Lemma 2.5, one has
$$\overrightarrow{b_{2^dn}}
={\rm diag}(J_{\lambda_1}^d, A_{1,m}^d)\overrightarrow{b_n},$$
where $J_{\lambda_1}$ is a Jordan block of order $r$, corresponding
to an eigenvalue $\lambda_1$ with $|\lambda_1|\textgreater1$.
Looking at the $r$-th component of the vector $\overrightarrow{b_n}$,
we have
$$(\overrightarrow{b_{2^dn}})_r=\lambda_1^d(\overrightarrow{b_n})_r.$$
If $(\overrightarrow{b_n})_r\ne0$, then
$$\lim_{d\rightarrow \infty}|(\overrightarrow{b_{2^dn}})_r|=\infty,$$
which contradicts with the boundedness of $\overrightarrow{b_{2^dn}}$.
Hence we must have
$(\overrightarrow{b_n})_r=0$.

By (\ref{eq1}), there exist complex numbers $c_1,\dots,c_{m-1}$ such that
\begin{align}\label{eq3.1}
\sum_{i=1}^{m-1}c_{i}t_m(n-i)=(\overrightarrow{b_n})_r=0
\end{align}
for any $n\ge m$. In fact, $(c_1,c_2,...,c_{m-1})$ is the $r$-th row vector of $P$.

Since $P$ is invertible, the coefficient vector $(c_1,\dots,c_{m-1})$ is nonzero.
If exactly one of the coefficients is nonzero, say $c_j$, then \eqref{eq3.1}
reduces to $c_j t_m(n-j)=0$ for all $n\ge m$, which implies that $t_m(k)=0$
for all sufficiently large $k$. However,
$$F_m(x)=\prod_{i=0}^\infty (1-x^{2^i})^m.$$
So $F_m(x)$ has infinitely many distinct zeroes on the unit circle. Since $F_m(0)=1\neq0$, $F_m(x)$
cannot be a zero polynomial. This shows that the case of exactly one nonzero coefficient cannot occur.
Therefore there exist at least two distinct indices $1\le i,j\le m-1$ with $c_i\ne 0$ and $c_j\neq0$.
Let $i_0$ and $j_0$ be such minimal index and maximal index, respectively. Then $1\le i_0<j_0\le m-1$,
$c_{i_0}\ne 0, c_{j_0}\ne 0$ and $c_i=0$ for all $i\in [1,i_0-1]\cup[j_0+1,m-1]$. Then (3.1)
becomes the following relation.
\begin{align}\label{eq3.2}
\sum_{i=i_0}^{j_0}c_{i}t_m(n-i)=0 \ \forall \ n\ge j_0.
\end{align}

By using Lemma \ref{lemma2.11}, we have $t_m(n)=\alpha_n+\beta_n$, where $\{\alpha_n\}_{n=1}^\infty$
is periodic and $\{\beta_n\}_{n=0}^\infty$ satisfies $\lim_{n\rightarrow \infty}|\beta_n|=0$. Then
the set of limit point of $\{\alpha_n\}_{n=0}^\infty$ is equal to the set of limit point
of $\{t_m(n)\}_{n=0}^\infty$, and we denote this set by $A$. Since $\{t_m(n)\}_{n=0}^\infty\in\mathbb{Z}$,
for $n$ large enough, we have $A\subseteq\mathbb{Z}$. Since $\{\alpha_n\}_{n=0}^\infty$ is periodic,
each of its values occurs for infinitely many times and therefore itself is a limit point of the sequence.
Hence $\{\alpha_n\}_{n=0}^\infty \subseteq A \subseteq \mathbb{Z}$. Thus $\beta_n=t_m(n)-\alpha_n$ is
the difference of two integers and so is an integer. Notice that $\beta_n\to 0$,
the only way an integer sequence can converge to zero is to be eventually identically zero.
Thus there exists an integer $N\ge 0$ such that $\beta_n=0$ for all integers $n\ge N$. Thus
\begin{align}\label{eq3.3}
t_m(n)=\alpha_n \ \forall \ n\ge N.
\end{align}

If $N=0$, then $t_m(n)=\alpha_n$ for all integers $n\ge 0$ and so $\{t_m(n)\}_{n=0}^\infty$
is a periodic sequence. In what follows, we assume that $N\ge 1$. In this case, we show that
$\{t_m(n)\}_{n=0}^\infty$ is still a periodic sequence. This will be done in the following.

For any integer $n\ge N+j_0$, by \eqref{eq3.3} we have $t_m(n-i)=\alpha_{n-i}$ for any
$i_0\le i\le j_0$. So by \eqref{eq3.2}, we have
\begin{align}\label{eq3.4}
\sum_{i=i_0}^{j_0}c_i\alpha_{n-i}=0 \ \forall \ n\ge N+j_0.
\end{align}
Let $q$ be the smallest positive period of $\alpha_n$. For any $l\in [1,N]$, there exists
an integer $u_\ell$ such that $N+j_0-\ell+u_\ell q\ge N+j_0$. So picking $n=N+j_0-\ell+u_\ell q$
in \eqref{eq3.4} yields that
\begin{align*}
\sum_{i=i_0}^{j_0} c_i\alpha_{N+j_0-\ell+u_\ell q-i}=0,
\end{align*}
which implies that
\begin{align}\label{eq3.5}
\sum_{i=i_0}^{j_0} c_i\alpha_{N+j_0-\ell-i}=0.
\end{align}

We {\sc claim} that for any integer $\ell'\in[1,N]$, if $t_m(n)=\alpha_{n}$
for all $n\ge\ell'$, then $t_m(\ell'-1)=\alpha_{\ell'-1}$. Letting
$\ell:=N-\ell'+1$ in \eqref{eq3.5} gives that
\begin{align*}
\sum_{i=i_0}^{j_0} c_i\alpha_{\ell'-1+j_0-i}=0,
\end{align*}
which implies that
\begin{align}\label{eq3.6}
\alpha_{\ell'-1}=-\frac{1}{c_{j_0}}\sum_{i=i_0}^{j_0-1} c_i\alpha_{\ell'-1+j_0-i}.
\end{align}
Since $\ell'-1+j_0-i\ge \ell'$ for $i_0\le i\le j_0-1$, we have
$$t_m(\ell'-1+j_0-i)=\alpha_{\ell'-1+j_0-i}.$$
Then by \eqref{eq3.6}, we deduce that
\begin{align}\label{eq3.7}
\alpha_{\ell'-1}=-\frac{1}{c_{j_0}}\sum_{i=i_0}^{j_0-1} c_it_m(\ell'-1+j_0-i).
\end{align}
Since $\ell'-1+j_0\ge j_0$, picking $n:=\ell'-1+j_0$ in \eqref{eq3.2} tells us that
$$\sum_{i=i_0}^{j_0}c_it_m(\ell'-1+j_0-i)=0.$$
This yields that
\begin{align}\label{eq3.8}
t_m(\ell'-1)=-\frac{1}{c_{j_0}}\sum_{i=i_0}^{j_0-1} c_it_m(\ell'-1+j_0-i).
\end{align}
Then comparing \eqref{eq3.7} and \eqref{eq3.8}, one arrives that
$t_m(\ell'-1)=\alpha_{\ell'-1}$. The claim is proved.

In conclusion, by the claim, we can deduce that $t_m(n)=\alpha_n$ for all
$0\le n\le N-1$. Moreover, this together with (3.3) implies that  $t_m(n)=\alpha_n$
for all $n\ge 0$. Therefore $t_m(n)$ is a periodic sequence with $q$ as its smallest
period.

Since $q$ is the smallest period of the sequence $\{t_m(n)\}_{n=0}^\infty$, one has
\begin{align*}
& (1-x^q)\prod_{i=0}^\infty(1-x^{2^i})^m\\
=& (1-x^q)\sum_{n=0}^{\infty}t_m(n)x^n\\
=&\sum_{n=0}^{q-1}t_m(n)x^n+\sum_{n=q}^{\infty}(t_m(n)-t_m(n-q))x^n\\
=&\sum_{n=0}^{q-1}t_m(n)x^n\\
:=& h(x),
\end{align*}
where $h(x)\in\mathbb{Z}[x]$
is a polynomial of degree $\deg(h(x))\le q-1$.

By Lemma \ref{lem7}, we know that
$$F_m(x)=\prod_{i=0}^\infty(1-x^{2^i})^m$$
is continuous on $[0,1)$. Let $a\ge 0$ be an integer such that
$$(1-x)^a|h(x)$$
and
$$(1-x)^{a+1}\nmid h(x).$$
One may write
$$h(x)=(1-x)^ah_1(x),$$
where $h_1(x)\in\mathbb{Z}[x]$ with
\begin{align}\label{eq3.9}
h_1(1)\not=0.
\end{align}

On the one hand, for any $x\in[0,1)$, one has
\begin{align*}
0<& \frac{(1-x^q)F_m(x)}{(1-x)^a}\\
=& (1-x^q)\prod_{i=a}^\infty(1-x^{2^i})^m\cdot \prod_{i=0}^{a-1}\frac{(1-x^{2^i})^m}{1-x}\\
\le & (1-x^q)\prod_{i=0}^{a-1}\frac{1-x^{2^i}}{1-x}.
\end{align*}
Taking $x\rightarrow 1^-$, we obtain that
$$\lim_{x\rightarrow 1^-}(1-x^q)\prod_{i=0}^{a-1}\frac{1-x^{2^i}}{1-x}
=\lim_{x\rightarrow 1^-}(1-x^q)\prod_{i=0}^{a-1}\prod_{j=0}^{i-1}(1+x^{2^j})
=0,$$
and so
$$\lim_{x\rightarrow 1^-}\frac{(1-x^q)F_m(x)}{(1-x)^a}=0.$$

On the other hand, we have
$$
\lim_{x\rightarrow 1^-}\frac{(1-x^q)F_m(x)}{(1-x)^a}
=\lim_{x\rightarrow 1^-} h_1(x)=h_1(1).
$$
Hence we can conclude that $h_1(1)=0$, which contradicts with (\ref{eq3.9}).
Therefore the sequence $\{t_m(n)\}_{n=0}^\infty$ must be unbounded.

This finishes the proof of Theorem 1.1. \hfill$\Box$\\

	\end{document}